\documentclass[a4paper,11pt,reqno,oneside,abstracton]{amsart}

\usepackage{scalerel,stackengine,relsize}
\stackMath
\usepackage{amssymb}
\usepackage{amsfonts}
\usepackage{mathrsfs}
\usepackage{bbm}
\usepackage{ifthen}
\usepackage{tikz}
\usepackage{tikz-3dplot}
\usepackage{bm}
\usepackage[hyphens]{url} 
\usepackage{amsmath,amssymb,amsfonts,amsthm,amsopn,dsfont,mathrsfs}
\usepackage{esint}
\usepackage{graphicx}
\usepackage{enumitem}
\usepackage{bigints}
\usepackage{fancyhdr}
\usepackage{hyperref}
\usepackage[T1]{fontenc}
\usepackage{mathtools}
\usepackage[utf8]{inputenc}
\usepackage{verbatim}
\usepackage[a4paper, total={6in, 9in}]{geometry}

\theoremstyle{plain}
\newtheorem{thm}{Theorem}[section]
\newtheorem*{thm*}{Theorem}
\newtheorem*{conj*}{Conjecture}

\newtheorem{lem}{Lemma}[section]
\newtheorem{cor}{Corollary}[section]
\theoremstyle{definition}

\newtheorem{dfn}{Definition}[section]

\theoremstyle{remark}
\newtheorem{rem}{Remark}[section]
\numberwithin{equation}{section}
\newtheorem*{ackn*}{Acknowledgements}

\newcommand\blfootnote[1]{%
  \begingroup
  \renewcommand\thefootnote{}\footnote{#1}%
  \addtocounter{footnote}{-1}%
  \endgroup
}
\newcommand{\intprod}{\mathbin{\mathpalette\dointprod\relax}}
\newcommand{\dointprod}[2]{%
  \raisebox{\depth}{\scalebox{1}[-1]{$#1\lnot$}}}
\title{Free boundary Hamiltonian stationary Lagrangian discs in $\mathbb{C}^2$}
\author{Filippo Gaia}

\begin{document}
\date{\today}

\begin{abstract}
Let \(\Omega \subset \mathbb{C}^2\) be a smooth domain. We establish conditions under which a weakly conformal, branched \(\Omega\)-free boundary Hamiltonian stationary Lagrangian immersion $u$ of a disc in $\mathbb{C}^2$ is a \(\Omega\)-free boundary minimal immersion. We deduce that if \(u\) is a weakly conformal, branched \(B_1(0)\)-free boundary Hamiltonian stationary Lagrangian immersion of a disc with Legendrian boundary data, then \(u(D^2)\) must be a Lagrangian equatorial plane disc. We also present examples of \(\Omega\)-free boundary Hamiltonain stationary discs, demonstrating the optimality of our assumptions.
\end{abstract}
\maketitle
\blfootnote{\textup{2020} \textit{Mathematics Subject Classification}: \textup{53D12, 53C42, 53C24, 58E12}}
\section{Introduction}
In this work, we study Lagrangian surfaces in $\mathbb{C}^2$ that are stationary points of the area functional with respect to Hamiltonian variations preserving the boundary of the surface within the boundary of a smooth open domain \(\Omega\) in \(\mathbb{C}^2\). Our main result is the following.
\begin{thm}\label{thm: HSLFB-implies-minimal}
    Let $\Omega\subset \mathbb{C}^2$ be an open subset with smooth boundary. 
    Let $u\in C^1\cap W^{2,1}(\overline{D}^2, \mathbb{C}^2)$ be a weakly conformal, Lagrangian immersion away from finitely many  branch points in $D^2$, with continuous Lagrangian angle $\overline{g}$.
    Assume that $u$ is $\Omega$-free boundary Hamiltonian stationary with the localization property. Moreover assume that
    \begin{align}\label{eq: intro-Legendrian-condition}
        \partial_\tau u\cdot I(N\circ u)=0\text{ on }\partial D^2,
    \end{align}
    where $I$ denotes the complex multiplication in $\mathbb{C}^2$ and $N$ denotes the outer normal vector of $\Omega$.
    Then $u$ is a calibrated, branched $\Omega$-free boundary minimal immersion. 
\end{thm}
In the special case $\Omega=B_1(0)$ we deduce the following rigidity result.
\begin{thm}\label{thm: main-thm}
    Let $u\in C^1\cap W^{2,1}(\overline{D}^2, \mathbb{C}^2)$ be a weakly conformal, Lagrangian immersion away from finitely many branch points in $D^2$, with continuous Lagrangian angle $\overline{g}$.
    Assume that $u$ is $B_1(0)$-free boundary Hamiltonian stationary with the localization property and that \eqref{eq: intro-Legendrian-condition} is satisfied. Then $u(D^2)$ is a flat equatorial Lagrangian disc (i.e. it is the intersection of $B_1(0)$ with a Lagrangian 2-plane passing through the origin).
\end{thm}
For the definition of free boundary Hamiltonian stationary map with the localization property we refer to  Definition \ref{dfn: localization-property}. The two Theorems are proved in Section 3.\\
Theorem \ref{thm: HSLFB-implies-minimal} is optimal in the following sense:
if $\overline{g}$ is not assumed to be continuous, there are examples of conformal, $B_1(0)$-free boundary Hamiltonian stationary Lagrangian maps with the localization property in $C^1\cap W^{2,1}(\overline{D}^2, \mathbb{C}^2)$, which are immersions away from a single point but are not minimal. We will see in Section 4.1 that the Schoen-Wolfson cones constitute a family of such examples.
Additionally, in Section 4.2 we will discuss examples of smooth domains $\Omega$ and of smooth, conformal $\Omega$-free boundary Hamiltonian stationary Lagrangian immersions with the localization property not satisfying \eqref{eq: intro-Legendrian-condition} which are not minimal.\\
The proof of the two Theorems is based on the observation that if $u$ is a $\Omega$-free boundary Hamiltonian stationary Lagrangian immersion away from finitely many branch points and singular points, satisfying \eqref{eq: intro-Legendrian-condition}, then
\begin{align}\label{eq: eq-condition-thm-2}
    (N\circ u)\wedge\partial_\nu u=0\quad\text{ and }\quad
        i\overline{g}\partial_\nu g=0\text{ on }\partial D^2
    \end{align}
(see Theorem \ref{thm: properties of HSL finite sing}).\\
The free boundary problem for smooth Lagrangian surfaces whose boundary lies either in a minimal Lagrangian submanifold or in a complex hypersurface in an $2n$-dimensional Calabi-Yau manifold (for any $n$) has been studied by R. Schoen in \cite{Schoen-special-lagrangian}. In particular he showed that if $\Sigma$ is a free boundary Lagrangian stationary submanifold, then the conormal vector of $\Sigma$ along $\partial \Sigma$ is orthogonal to the constraint manifold, moreover in the first case its Lagrangian angle is constant on $\partial \Sigma$, while in the second case the normal derivative of the Lagrangian angle vanishes on $\partial \Sigma$. In both cases he deduced that $\Sigma$ is special Lagrangian and therefore free boundary minimal and calibrated.\\
Regarding Theorem \ref{thm: main-thm}, we first observe that the condition $\partial_\tau u\cdot I(N\circ u)=0$ in this case corresponds to requiring that the boundary curve $u\vert_{\partial D^2}$ is \textit{Legendrian} with respect to the standard contact structure of $\mathbb{S}^3$. A rigidity result analogous to Theorem \ref{thm: main-thm} (without the Lagrangian constraint or the Legendrian assumption) has been established by J. Nitsche in \cite{Nitsche} (for $n=3$) and by A. Fraser and R. Schoen in \cite{Fraser-Schoen} (for any $n$ and in any space form): they showed that for any branched minimal immersion $u: D^2\to B_1^n(0)$ such that $u(D^2)$ meets $ \partial B_1^n(0)$ orthogonally, $u(D^2)$ is an equatorial plane disc. Similar rigidity results for capillary stable surfaces, without a priori assumptions on the topology of the surfaces, have been obtained by A. Ros and R. Souam in \cite{Ros-Souam} (see also \cite{Souam}).\\
More recently, M. Li, G. Wang and L. Weng showed in \cite{LWW} that if $u$ is a branched, minimal, Lagrangian immersion with Legendrian capillary boundary on $\mathbb{S}^3$, then $u(D^2)$ is an equatorial plane disc. Their work was motivated by the study of Lagrangian surfaces in a symplectic manifold $M$ with $\omega$-convex boundary (see Section 1.5 in \cite{EG} for the definition and some examples). In particular they suggested to look at free boundary (or capillary boundary) Lagrangian surfaces in such manifolds, focusing on the case of Lagrangian surfaces with Legendrian boundary in $B_1(0)$. Further rigidity results in this setting have been obtained by Y. Luo and L. Sun in \cite{YS}.\\
So far, the results available in the literature have concerned smooth free boundary Hamiltonian stationary Lagrangian surfaces. However, R. Schoen and J. Wolfson showed in \cite{SW} that even $W^{1,2}$ minimizers of the area among Lagrangian maps are smooth only away from a locally finite set of points, consisting of branch points and singular points having a Schoen--Wolfson cone as tangent map. Examples of $B_1(0)$-free boundary Hamiltonian stationary surfaces with singularities are given for instance by the Schoen Wolfson cones (see Lemma \ref{lem: SW-cones}).
We also remark that any weakly conformal, Hamiltonian stationary Lagrangian map in $W^{1,2}$ with isolated singularities lies locally in $C^{1,\sqrt{2}-1}$ (see Lemma V.3 in \cite{GOR}). Therefore it seems natural to study the $\Omega$-free boundary Hamiltonian stationary problem in the class of $C^1$ Lagrangian maps which are immersions outside of isolated singular points and branch points in $D^2$. For our results we assumed $u$ to be of class $C^1$ up to the boundary, in order to be able to exclude branch points or singularities on $\partial D^2$ and to make sense of the trace of $g$ on $\partial D^2$. The regularity assumption $u\in W^{2,1}(D^2, \mathbb{C}^2)$ seems to be necessary to deduce the Euler-Lagrangian equation \eqref{eq: EL-in-Omega} from the condition that $u$ is $\Omega$-free boundary Hamiltonian stationary with the localization property (see the proof of Lemma \ref{lem: properties-of-g}).\\
The localization property, defined in Definition \ref{dfn: localization-property}, appears to be a natural assumption for applying the mapping approach to study variations in the target. It has been introduced in \cite{targetharm} by T. Rivière to study target harmonic maps.\\
We conclude the Introduction by listing some related open questions:
\begin{enumerate}
\item Is it possible to show the same results under weaker regularity assumptions on $u$ and $g$? Ideally one would like to assume only $u\in W^{1,2}$ and $g$ continuous. Notice that in the classical (non-Lagrangian) case, $B_1(0)$-free boundary maps in $W^{1,2}$ (i.e. harmonic extensions of half-harmonic maps in $H^\frac{1}{2}(\partial D^2, \mathbb{S}^n)$) are smooth up to the boundary thanks to \cite{3-commutators}, so that they define possibly branched free boundary minimal immersions in the smooth sense. Then the rigidity result of Nitsche and Fraser-Schoen applies.
\item Are there examples of $B_1(0)$-free boundary Hamiltonian stationary Lagrangian surfaces with more than one singularity? In this regard, we remark that in \cite{GOR} we constructed Hamiltonian stationary Lagrangian surfaces with multiple singularities, which for real-analytic boundary data are free boundary Hamiltonian stationary Lagrangian surfaces for a holomorphic curve as constraint manifold.
\item Can condition \eqref{eq: intro-Legendrian-condition} in Theorem \ref{thm: main-thm} be weakened or removed?
\end{enumerate}

\textbf{Acknowledgements.}
    I'm very grateful to Tristan Rivi\`ere for his constant guidance and support and to Gerard Orriols for the stimulating and instructive discussions on the subject of Hamiltonian stationary Lagrangian maps.
\section{Preliminaries}\label{sec: preliminaries}
Let $\omega=dx_1\wedge dy_1+dx_2\wedge dy_2$ be the standard symplectic form on $\mathbb{C}^2$. A surface $\Sigma\subset \mathbb{C}^2$ is \textit{Lagrangian} if $\omega\vert_\Sigma=0$.
Analogously, a map $u\in W^{1,2}(D^2,\mathbb{C}^2)$ is \textit{Lagrangian} if $u^\ast\omega=0$ a.e..\\
A map $u\in W^{1,1}(D^2,\mathbb{C}^2)$ is said to be \textit{weakly conformal} if $\langle \partial_xu, \partial_yu\rangle=0$ and $\lvert\partial_xu\rvert^2=\lvert\partial_yu\rvert^2$ a.e..\\
If $u\in W^{1,2}(D^2,\mathbb{C}^2)$ is Lagrangian and weakly conformal, there exists a measurable map $g: D^2\to\mathbb{S}^1$ such that
\begin{align}\label{eq: Lagrangian angle}
    u^\ast(dz_1\wedge dz_2)=e^{2\lambda}\overline{g}\, dx\wedge dy \text{ a.e.},
\end{align}
where $e^{2\lambda}=\lvert\partial_x u\rvert^2=\lvert\partial_y u\rvert^2$
(see for instance p. 3 in \cite{SWsurvey}). The map $\overline{g}$ is called the \textit{Lagrangian angle} of $u$. If a weakly conformal Lagrangian map $u$ has constant Lagrangian angle, $u$ is said to be \textit{special Lagrangian}. If $u\in C^1(\overline{D^2},\mathbb{C}^2)$ is a special Lagrangian branched immersion, then $u$ is a \textit{calibrated} minimal branched immersion (see p. 4 in \cite{SWsurvey}). By a \textit{minimal branched immersion} we mean here a non-constant smooth map which is weakly conformal and harmonic, so that it parametrizes a smooth minimal surface outside of isolated points.\\
It will be convenient to identify $\mathbb{C}^2$ with the algebra of quaternions $\mathbb{H}$ with basis elements $1,I,J,K$. In this identification, $I$ corresponds to the complex multiplication.
For any weakly conformal, Lagrangian $u\in W^{1,2}(D^2)$, \eqref{eq: Lagrangian angle} implies
\begin{align}\label{eq: Lag-angle-def}
    \frac{1}{r}\partial_\theta u=-\overline{g}J\partial_r u\text{ a.e.,}
\end{align}
so that
\begin{align}\label{eq: structural-equation}
    \operatorname{div}({g}\nabla u)=0.
\end{align}
Let $u\in C^1(\overline{D}^2,\mathbb{C}^2)$ be weakly conformal; we will say that $u$ is \textit{Lagrangian stationary}
if for any $\tilde{u}\in C^1((-\varepsilon,\varepsilon),C^2(\overline{D}^2,\mathbb{C}^2))$ with
\begin{enumerate}
    \item $\tilde u(0,\cdot)=u$
    \item $\tilde u(t,\cdot)$ is Lagrangian for any $t$
    \item $\tilde u(t,\cdot)=u$ outside of a compact set $K\subset D^2$ for any $t$
\end{enumerate}
we have 
\begin{align}\label{eq: stationarity-def}
    \frac{d}{dt}\bigg\vert_{t=0}\frac{1}{2}\int_{D^2}\lvert\nabla\tilde u(t,x)\rvert^2 dx=0.
\end{align}
If $u$ is a conformal immersion, this corresponds to requiring that $u(D^2)$ is stationary for the area with respect to compactly supported Lagrangian variations.\\
Let $\Omega\subset \mathbb{C}^2$ be an open subset with smooth boundary. 
We will say that $u$ is \textit{$\Omega$-free bondary Lagrangian stationary} if \eqref{eq: stationarity-def} holds for any $\tilde u$ satisfying $(1)$ and $(2)$ above and such that $\frac{\partial}{\partial t}\big\vert_{t=0} \tilde u(t,x)$ is tangent to $\partial\Omega$ for any $x\in \partial D^2$.\\
We will also need the following definition\footnote{The definition of the localization property is motivated by the concept of target harmonic maps introduced in \cite{targetharm}, and seems suitable also to study the problem with lower regularity assumptions (with the property holding for almost any domain as in \cite{targetharm}).}:
\begin{dfn}\label{dfn: localization-property}
A Lagrangian map $u\in C^1(\overline{D^2}, \mathbb{C}^2)$ with $u(\partial D^2)\subset\partial \Omega$ is \textit{$\Omega$-free boundary Hamiltonian stationary with the localization property} if for every smooth, relatively open domain $\omega\subset \overline{D^2}$,
\begin{align}\label{eq: hsllp}
        \int_\omega \langle d (I(\nabla f)\circ u); du\rangle=0
\end{align}
for any smooth function $f:\mathbb{C}^2\to \mathbb{R}$ supported away from $u(\partial\omega\cap D^2$) such that $I\nabla f(x)\cdot N(x)=0$ for any $x$ in a neighbourhood of $u(\overline{\omega}\cap\partial D^2)$ in $\partial \Omega$, where $N$ denotes the outer normal vector of $\Omega$.\\
If \eqref{eq: hsllp} holds for any smooth $\omega\Subset D^2$, $u$ is \textit{Hamiltonian stationary with the localization property}.
\end{dfn}
Simple examples of $\Omega$-free boundary Hamiltonian stationary maps with the localization property can be constructed as follows.
Let $D\subset \mathbb{C}$ be a smooth domain. Let $i: D\to \mathbb{C}^2$ be the embedding of $D$ resulting from identifying $\mathbb{C}$ with $\mathbb{R}\times\{0\}\times\mathbb{R}\times\{0\}\subset\mathbb{C}^2$.
Let $A\in U(2)$ and set $u:=A\circ i$.
Let $\Omega\subset\mathbb{C}^2$ be a smooth domain such that $u(\partial D)\subset \partial \Omega$ and such that on $\partial D$ there holds $\partial_\nu u\wedge (N\circ u)=0$ (where $N$ denotes the outer normal vector of $\Omega$). Then $u$ is a conformal, $\Omega$-free boundary Hamiltonian stationary Lagrangian map with the localization property.
Further examples are discussed in Section 4.

By the next Lemma, this property is satisfied by any $\Omega$-free boundary Lagrangian stationary map.
\begin{lem}\label{lem: equivalence of definitions}
Let $u\in C^1(\overline{D}^2, \mathbb{C}^2)$ be a Lagrangian map. If 
$u$ is ($\Omega$-free boundary) Lagrangian stationary, then 
$u$ is ($\Omega$-free boundary) Hamiltonian stationary with the localization property.
\end{lem}
\begin{proof}
    Assume that $u$ is a ($\Omega$-free boundary)  branched Lagrangian stationary immersion. Let $\omega\Subset D^2$ be a subdomain ($\omega\subset \overline{D^2}$ in the case of $\Omega$-free boundary maps) and let $f$ be a smooth map on $\mathbb{C}^2$ supported away from $u(\partial \omega\cap D^2)$ (in the $\Omega$-free boundary case we also require that $I\nabla f(x)\cdot N(x)=0$ for any $x$ in a neighbourhood of $u(\overline{\omega}\cap \partial D^2)$ in $\partial \Omega$).
    Let $\tilde{u}(t,x)$ be the solution of
    \begin{align}
        \begin{cases}
            \partial_t \tilde{u}(t,x)=I\nabla f(\tilde u(t,x))\\
            \tilde{u}(0,x)=u(x)
        \end{cases}
    \end{align}
    for any $x\in \overline{\omega}$ and set $\tilde{u}(t,x)=u(x)$ for any $x\in \overline{D}^2\smallsetminus\omega$. We claim that $\tilde{u}$ is a variation of $u$ through Lagrangian maps.
    In fact
    \begin{align}
        \mathcal{L}_{I\nabla f}\omega=d(I\nabla f\intprod\omega)+I\nabla f\intprod d\omega=0,
    \end{align}
    as $I\nabla f\intprod \omega=-df$ and $d\omega=0$. Since $u^\ast\omega=0$ we conclude that $\tilde u(\cdot, t)^\ast\omega=0$ for any $t$. Moreover in the free boundary case
    $\frac{\partial}{\partial t}\big\vert_{t=0}\tilde{u}(t,x)=I\nabla f(u(x))\in T_{u(x)}\partial\Omega$
     for any $x\in \partial D^2$. Thus $\tilde{u}$ is an admissible variation, and as $u$ is ($\Omega$-free boundary) Lagrangian stationary, we have
    \begin{align}
        \int_{\omega} \langle d(I(\nabla f)\circ u); du\rangle=\frac{d}{dt}\bigg\vert_{t=0}\frac{1}{2}\int_{D^2}\lvert \nabla \tilde{u}(x,t)\rvert^2\, dx=0.
    \end{align}
    This shows that $u$ is ($\Omega$-free boundary) Hamiltonian stationary with the localization property.
\end{proof}
Notice that if $u\in C^1(\overline{D^2},\mathbb{C}^2)$ is a conformal immersion, then $g$ is continuous in $\overline{D}^2$ since $\overline{g}=e^{-2\lambda}\det(\nabla u)$.
\begin{dfn}\label{def: singular-branch}
    Let $u\in C^1(\overline{D^2},\mathbb{C}^2)$ be a weakly conformal Lagrangian map such that $\nabla u$ vanishes at isolated points.\\
    We say that $p\in \overline{D^2}$ is a \textit{branch point} of $u$ if $\nabla u(p)=0$ and the Lagrangian angle $\overline{g}$ of $u$ is continuous at $p$.\\
    We say that $p\in \overline{D^2}$ is a \textit{singular point} of $u$ if its Lagrangian angle $\overline{g}$ is not continuous at $p$.
\end{dfn}
In particular, if $p$ is a singular point of $u$, then $\nabla u(p)=0$.
\begin{lem}\label{lem: properties-of-g}
    Let $u\in C^1\cap W^{2,1}(\overline{D}^2, \mathbb{C}^2)$ be 
    a weakly conformal branched immersion away from finitely many branch points and singular points $q_1,...,q_M\in D^2$, with Lagrangian angle $\overline{g}$. Then $\overline{g}\in C^\infty_{\text{loc}}(D^2\smallsetminus\{q_1,...,q_M\})$ and satisfies
    \begin{align}\label{eq: EL-in-Omega}
         \operatorname{div}(\overline{g}\nabla g)=0\quad\text{ and }\quad\operatorname{div}(i\overline{g}\nabla^\perp g)=0\text{ in }D^2\smallsetminus\{q_1,..., q_M\}.
    \end{align}
    If $u$ has no singularities, there exists an harmonic function $\beta$ on $D^2$, continuous up to the boundary, such that $g=e^{i\beta}$.\\
    If we also assume that $g$ lies in $W^{1,1}(D^2)$, we have
    \begin{align}\label{eq: EL-W11}
        \operatorname{div}(\overline{g}\nabla g)=\sum_{i=1}^M\alpha_i\delta_{q_i}\quad\text{ and }\quad\operatorname{div}(i\overline{g}\nabla^\perp g)=2\pi\sum_{i=1}^Md_i\delta_{q_i}\text{ in }D^2
    \end{align}    
    for $\alpha_1,...,\alpha_M\in \mathbb{R}$ and integers $d_1,...,d_M$.
\end{lem}
\begin{proof}
    Assume first that $u$ has no singularities. Let $\{p_i\}_{i\in I}$ denote the 
    finitely many branch points of $u$ in $D^2$ and let $\Xi:=D^2\smallsetminus\bigcup_{i\in I}\{p_i\}$. Notice that since $u\in C^1(\overline{D^2},\mathbb{C}^2)$, $g$ is continuous on $\Xi$. Since $\nabla u\in W^{1,1}(D^2)$, and $\lvert\nabla u\lvert$ is bounded from above and locally bounded away from zero in $\Xi$, $\overline{g}=\det\left(\frac{\nabla u}{\lvert \nabla u\rvert}\right)\in W^{1,1}_{\text{loc}}(\Xi)$.
    Let $\omega\Subset \Xi$ be a smooth subdomain and let $f$ be a smooth function on $\mathbb{C}^2$ supported away from $u(\partial \omega)$. If $u$ is Hamiltonian stationary with the localization property, then we have
    \begin{align}\label{eq: computation-Lemma-I}
        0=&\int_{\omega}\langle du; d(I(\nabla f)\circ u)\rangle=-\int_{\omega}\langle i\overline{g}dg\cdot du, (\nabla f)\circ u\rangle
        =-\int_{\omega} i\overline{g}dg\cdot d(f\circ u),
    \end{align}
    where the second equality follows from \eqref{eq: structural-equation}.
    For any $p\in \Xi$ let $\omega_p$ be a subdomain as above containing $p$ and such that $u$ restricted to $\omega_p$ is a $C^1$ diffeomorphism to its image.    
    Notice that any $\varphi\in C_c^1(\omega_p)$ can be written as $\varphi=f_\varphi\circ u$ for some $C^1$ function $f_\varphi$ on $\mathbb{C}^2$ supported away from $u(\partial \omega_p)$.
    Approximating $f_\varphi$ in $C^1$ by smooth functions with the same properties we obtain
    \begin{align}
        -\int_\omega i\overline{g}dg\cdot d\varphi=-\int_\omega i\overline{g}dg\cdot d(f_\varphi\circ u)=0.
    \end{align}
    Therefore
    \begin{align}\label{eq: g-S1-harmonic}
        \operatorname{div}(\overline{g}\nabla g)=0
    \end{align}
    on $\omega_p$, and since the argument holds for any $p\in \Xi$ we conclude that \eqref{eq: g-S1-harmonic} holds in $\Xi$.
    Since $g$ is continuous, there exists $\beta\in C^0(\overline{D^2})$ (unique up to addition of a constant in $2\pi\mathbb{Z}$) such that $g=e^{i\beta}$. By \eqref{eq: g-S1-harmonic}, $\beta$ satisfies
    $\Delta \beta=0$ in $\Xi$.
    Since $\beta$ is bounded and has isolated singularities, $\beta$ is harmonic in $D^2$. In particular, we have
    \begin{align}
        \operatorname{div}(i\overline{g}\nabla^\perp g)=-\operatorname{div}(\nabla^\perp \beta)=0\text{ in }D^2.
    \end{align}
    Next assume that $u$ is an immersion away from finitely many branch points and singularities $q_1,...,q_M\in D^2$.
    Arguing as in the first part of the proof we see that on any simply connected, smooth, open $\omega\subset\Psi:=D^2\smallsetminus\{q_1,...,q_M\}$ there exists an harmonic function $\beta$ such that $g=e^{i\beta}$ on $\omega$, so that $\operatorname{div}(\overline{g}\nabla g)=0$ and $\operatorname{div}(\overline{g}\nabla^\perp g)=0$ on $\Psi$.
    Arguing as in Lemma 2 in \cite{BBM} we obtain that 
    \begin{align}
        \operatorname{div}(ig\nabla\overline{g})=\sum_{i=1}^M\alpha_i\delta_{q_i}\text{ and }\operatorname{div}(i\overline{g}\nabla^\perp g)=2\pi\sum_{i=1}^Md_i\delta_{q_i}\text{ in }D^2
    \end{align}    
    for $\alpha_1,...,\alpha_M\in \mathbb{R}$ and integers $d_1,...,d_M$ (these are the degrees of the singular points, see \cite{BBM} for details).
\end{proof}
\begin{rem}\label{rem: Neumann-distribution}
    Let $u\in C^1\cap W^{2,1}(\overline{D^2},\mathbb{C}^2)$ be a weakly conformal, Lagrangian immersion away from finitely many singular points and branch points in $D^2$, with Lagrangian angle $\overline{g}$.
    Let $\omega\subset \overline{D^2}$ be a relatively open, simply connected, smooth subset such that $\overline{\omega}$ contains no singularity or branch point of $u$. Then $\overline{g}$ is continuous on $\overline{\omega}$ and it can be written as $\overline{g}=e^{-i\beta}$ for an harmonic function $\beta\in W^{1,1}(\omega)\cap C^1(\overline{\omega})$. Therefore there exists $G_\omega\in W^{1,1}(\omega)$ such that $-i\overline{g}\nabla g=\nabla \beta=\nabla^\perp G_\omega$ in $\omega$ (i.e. $G_\omega$ is the harmonic conjugate of $\beta$ in $\omega$).\\
    Notice that since $\beta$ is continuous on $\partial \omega$, for any $p\in (1,\infty)$ there holds
    $G_\omega\in L^p(\partial \omega)$ with $\lVert G_\omega\rVert_{L^p(\partial\omega)}\leq C(p)\lVert \beta\rVert_{L^p(\partial\omega)}$ (by continuity of the Hilbert transform in $L^p(\partial\omega)$).\\
    For any $\varphi\in C^1(\overline{\omega})$ there holds
    \begin{align}\label{eq: def-Neumann-der}
        \int_{\omega}i\overline{g}\nabla g\cdot \nabla \varphi=-\int_\omega \nabla^\perp G_\omega\cdot\nabla\varphi=\int_{\partial\omega}G_\omega\partial_\tau \varphi,
    \end{align}
    where $\tau$ denote the oriented unit tangent vector on $\partial \omega$.
\end{rem}
    Let $u\in C^1\cap W^{2,1}(\overline{D^2},\mathbb{C}^2)$ be a weakly conformal, Lagrangian immersion away from finitely many singular points and branch points in $D^2$, with Lagrangian angle $\overline{g}$.
    Let $\{\omega_i\}_{i=1}^Q$ be a covering of $\partial D^2$ by relatively open, simply connected, smooth subset of $\overline{D^2}$ containing no singular or branch point of $u$, as in Remark \ref{rem: Neumann-distribution}, and let $\omega_0\Subset D^2$ such that $\overline{D^2}=\bigcup_{i=0}^Q\omega_i$. Let $\{\psi_i\}_{i=0}^Q$ be a partition of unity subordinate to $\{\omega_i\}_{i=0}^Q$.
    For any $\varphi\in C^1(\partial D^2)$
    we define
    \begin{align}\label{eq: def-Neumann-distribution}
        \langle i\overline{g}\partial_\nu g, \varphi\rangle
        :=\sum_{i=1}^Q\int_{\omega_i}i\overline{g}\nabla g\cdot\nabla(\tilde{\varphi}\psi_i)
        =\sum_{i=1}^Q\int_{\partial \omega_i}G_{\omega_i}\partial_\tau(\varphi\psi_i),
    \end{align}
    where $\tilde\varphi$ is any $C^1$ extension of $\varphi$ in $D^2$.
    The following estimate holds:
    \begin{align}
      \lvert\langle i\overline{g}\partial_\nu g, \varphi\rangle\rvert\leq C\sum_{i=1}^Q \lVert G_{\omega_i}\rVert_{L^1(\partial\omega_i\cap \partial D^2)}\lVert \varphi\rVert_{C^1(\partial D^2)}, 
    \end{align}
    therefore $i\overline{g}\partial_\nu g$ defines a distribution on $\partial D^2$.
    If $\overline{g}\in C^1(\overline{D^2})$, then $\langle i\overline{g}\partial_\nu g, \varphi\rangle=\int_{\partial D^2}i\overline{g}\partial_\nu g\varphi$.
    The two expressions in \eqref{eq: def-Neumann-distribution} show that the definition of $i\overline{g}\partial_\nu g$ does not depend on the choice of the sets $\omega_i$, on the partition if unity or on the extension $\tilde\varphi$ of $\varphi$.\\
    Notice that if for any $i\in\{1,...,Q\}$ $G_{\omega_i}$ is constant on $\overline{ \omega_i}\cap\partial D^2$, then integration by parts implies that $i\overline{g}\partial_\nu g=0$.
    
\section{The main results}
Theorem \ref{thm: HSLFB-implies-minimal} and Theorem \ref{thm: main-thm} will be deduced from the following characterization result.
\begin{thm}\label{thm: properties of HSL finite sing}
    Let $\Omega\subset \mathbb{C}^2$ be an open subset with smooth boundary. 
    Let $u\in C^1\cap W^{2,1}(\overline{D}^2, \mathbb{C}^2)$ be a weakly conformal, Lagrangian immersion away from finitely many singular points and finitely many branch points in $D^2$, with Lagrangian angle $\overline{g}$.
    Assume that $u$ is $\Omega$-free boundary Hamiltonian stationary with the localization property. Assume in addition that for any $x\in \partial D^2$,
    \begin{align} \label{eq: assumption-Legendrian}
        \partial_\tau u(x)\cdot IN(u(x))=0.
    \end{align}
    Then
    \begin{align} \label{eq: orthogonality-prop}  (N\circ u)\wedge\partial_\nu u=0\text{ on }\partial D^2
    \end{align}
    and
    \begin{align} \label{eq: bdry-condition-g}
        i\overline{g} \partial_\nu g =0\text{ on }\partial D^2,
    \end{align}
    where $i\overline{g}\partial_\nu g$ is defined as in \eqref{eq: def-Neumann-distribution}.\\
    Conversely, if $\overline{g}$ lies in $W^{1,1}(D^2)$ and satisfies
    \begin{align}
        \operatorname{div}(i\overline{g}\nabla g)=0
    \end{align}
    as well as equations \eqref{eq: orthogonality-prop} and \eqref{eq: bdry-condition-g}, then $u$ is $\Omega$-free boundary Hamiltonian stationary with the localization property.
\end{thm}
\begin{proof}
    Assume first that $u$ is $\Omega$-free boundary Hamiltonian stationary Lagrangian with the localization property and that \eqref{eq: assumption-Legendrian} holds.
    Let $\omega\subset \overline{D^2}$ be a relatively open, simply connected, smooth domain such that $\overline{\omega}$ does not contain any singular or branch point of $u$. Then $g$ is continuous on $\overline{\omega}$.
    Let $f$ be a smooth function on $\mathbb{C}^2$ supported away from $u(\partial \omega \cap D^2)$ and such that $I\nabla f(x)\cdot N(x)=0$ for any $x$ in a neighbourhood of $u(\overline{\omega}\cap\partial D^2)$. Then we have
    \begin{align}\label{eq: computation-HSlp}
        0=&\int_{\omega}\langle du; d(I(\nabla f)\circ u)\rangle=\int_{\overline{\omega}\cap\partial D^2}\langle\partial_\nu u, I(\nabla f)\circ u\rangle+\int_{\omega}\langle \overline{g}dg\cdot d u, I(\nabla f)\circ u\rangle\\
        \nonumber
        =&\int_{\overline{\omega}\cap\partial D^2}\langle\partial_\nu u, I(\nabla f)\circ u\rangle-\int_\omega i\overline{g}dg\cdot d(f\circ u)\\
        \nonumber
        =&\int_{\overline{\omega}\cap\partial D^2}\langle\partial_\nu u, I(\nabla f)\circ u\rangle-\int_{\partial \omega}G_\omega \partial_\tau (f\circ u),
    \end{align}
    where the second equality follows from \eqref{eq: structural-equation} and the function $G_\omega$ in the last expression is the one introduced in Remark \ref{rem: Neumann-distribution}.
    Notice that by the assumptions on $f$, in the second term of the right hand side the integrand vanishes on $\partial \omega\cap D^2$.
    Thus for
    $\sigma:=\partial D^2\cap\overline{\omega}$,
    \begin{align} \label{eq: boundary-identity-I}
        \int_{\sigma} \langle\partial_{\nu}u, I(\nabla f)\circ u\rangle=\int_\sigma G_\omega\partial_\tau (f\circ u)
    \end{align}
    for any smooth $f$ supported away from $u(\partial\sigma)$ and such that $I\nabla f(x)\cdot N(x)=0$ for any $x$ in a neighbourhood of $u(\overline{\omega}\cap\partial D^2)$. In particular, \eqref{eq: boundary-identity-I} holds for any segment $\sigma\subset \partial D^2$ (for $f$ as above).\\
We claim that on $\partial D^2$ there holds $\partial_\nu u\cdot I(N\circ u)=0$.
To see this, let
\begin{align}
    \Sigma_+:=\left\{x\in \partial D^2\vert \partial_\nu u(x)\cdot IN(u(x))>0\right\}.
\end{align}
Notice that $\Sigma_+$ is an open subset of $\partial D^2$.
Moreover, for any $n\in \mathbb{N}$ set
\begin{align}
    \Sigma_+^\frac{1}{n}:=\left\{x\in \Sigma_+\vert\operatorname{dist}(x,\partial\Sigma_+)>\frac{1}{n}\right\}.
\end{align}
For any $n\in \mathbb{N}$ let $\varphi_n$ be a smooth non-negative function supported on $\Sigma_+$ and such that $\varphi_n\equiv 1$ on $\Sigma_+^\frac{1}{n}$ and let $f_n$ be a smooth function defined on $\mathbb{C}^2$, supported away from $u(\partial\omega\cap D^2)$ and such that
\begin{align}
    \begin{cases}
        f_n\equiv 0&\text{ on }\partial\Omega\\
        \nabla f_n(x)=\varphi_n(x)N(x)&\text{ on }\partial\Omega.
    \end{cases}
\end{align}
Then \eqref{eq: boundary-identity-I} implies that for any segment $\sigma$ as above
\begin{align}
    \int_{\sigma\cap \Sigma_+^\frac{1}{n}}\lvert \partial_\nu u\cdot I(N\circ u)\rvert\leq \int_{\sigma\cap\Sigma_+} \langle\partial_{\nu}u, I(\nabla f_n)\circ u\rangle=0.
\end{align}
As this holds for any $n\in \mathbb{N}$ and $\Sigma_+=\bigcup_{n\in \mathbb{N}}\Sigma_+^\frac{1}{n}$, we conclude that $\Sigma_+$ has measure zero. Similarly one can show that the set of points $x$ in $\partial D^2$ where $\partial_\nu u(x)\cdot IN(u(x))<0$ also has measure zero. This proves the claim.\\
Notice that by assumption \eqref{eq: assumption-Legendrian},
\begin{align*}
    \left\{N(u(x)), IN(u(x)), \partial_\tau u(x), I\partial_\tau u(x)\right\}
\end{align*}
is an orthogonal basis of $T_{u(x)}\mathbb{C}^2$ for any $x\in \partial D^2$.
For any $x\in \partial D^2$, $IN(u(x))\cdot\partial_\nu u(x)=0$ by the previous claim, $I\partial_\tau u\cdot\partial_\nu u=0$ as $u$ is Lagrangian and $\partial_\tau u\cdot\partial_\nu u=0$ as $u$ is conformal.
We conclude that $\partial_\tau u(x)\wedge N(u(x))=0$, this shows \eqref{eq: orthogonality-prop}.

Next we claim that $i\overline{g}\partial_{\nu}g=0$ on $\partial D^2$.
We will show that for any $x\in \partial D^2$, there exists $\omega$ as in \eqref{eq: boundary-identity-I} such that $\overline{\omega}$ contains a neighbourhood of $x$ in  $\partial D^2$ and $G_\omega$ is locally constant on $\partial D^2\cap \overline{\omega}$.\\
Assume first that $u$ is smooth. Notice that on $\partial D^2$
\begin{align}
    \partial_\tau u\cdot I(N\circ u)=-\overline{g}J\partial_\nu u\cdot I(N\circ u)=0
\end{align}
by \eqref{eq: orthogonality-prop}.
For any $x\in \partial D^2$ let $\omega$ be a domain as above with $x\in \omega$, on which $u$ is injective and such that $\overline{\omega}\cap \partial D^2$ is connected.
Let $\alpha$ be a smooth function on $\overline{\omega}\cap \partial D^2$ supported away from $\partial(\overline{\omega}\cap\partial D^2)$ and such that $\int_{\overline{\omega}\cap \partial D^2}\alpha=0$. Then there exists a smooth function $f$ on $\mathbb{C}^2$ such that $\nabla f\cdot IN=0$ in a neighbourhood of $u(\overline{\omega}\cap\partial D^2)$ in $\partial \Omega$ and $\partial_\tau(f\circ u)=\alpha$ on $\overline{\omega}\cap\partial D^2$. Plugging $f$ in \eqref{eq: boundary-identity-I} and varying $\alpha$ we obtain that $G_\omega$ is constant on $\overline{\omega}\cap\partial D^2$.
This argument, however, fails when $u$ is only of class $C^1\cap W^{2,1}(\overline{D}^2, \mathbb{C}^2)$.\\
For the general case we proceed as follows.
Given $x\in \partial D^2$ set $p:=u(x)$, $v=\frac{\partial_\tau u(x)}{\lvert\partial_\tau u(x)\rvert}$. Let $g_0\in \mathbb{S}^1\subset\mathbb{C}$ such that $v=g_0JN(p)$ (recall that $\partial_\tau u(x)\cdot N(u(x))=0$, $\partial_\tau u(x)\cdot IN(u(x))=0$ on $\partial D^2$). Let $P$ be the hyperplane passing through $p$ and orthogonal to $v$. Extend $N$ to a smooth vector field in a neighbourhood of $p$ (still denoted by $N$). Notice that for $\varepsilon>0$ sufficiently small, the flow of $g_0JN$ induces a diffeomorphism $\Psi$ from $B^P_\varepsilon(p)\times(-\delta, \delta)$ to a neighbourhood $U$ of $p$ (where $B^P_\varepsilon(p)$ denotes a ball in the plane $P$). If $\delta$ is chosen sufficiently small, $\Psi$ satisfies
\begin{enumerate}
    \item $\partial_t\Psi=g_0J(N\circ\Psi)$ ($t$ denotes the coordinate in $(-\delta, \delta)$);
    \item on the connected component $\sigma$ of $u^{-1}(U)\cap\partial D^2$ containing $x$, there holds
    \begin{align}
        \frac{\partial_\tau u(y)}{\lvert \partial_\tau u(y)\rvert}\cdot g_0JN>\frac{1}{2};
    \end{align}
    \item there holds
    \begin{align}
        (\Psi^{-1}\circ u)(\partial \sigma)\subset B^P_\varepsilon (p)\times\{-\delta, \delta\}.
    \end{align}
\end{enumerate}
Let $\omega$ be a smooth, simply connected, relatively open subset of $\overline{D^2}$ which contains no singular or branch point of $u$ and such that $x\in \omega$ and $\omega\cap\partial D^2=\sigma$.
For any $\alpha\in C^\infty_c((-\delta, \delta))$ with $\int_{-\delta}^\delta\alpha=0$ let $\beta\in C^\infty_c((-\delta, \delta))$ such that $\beta'=\alpha$. Set
\begin{align}
    \tilde\beta: B^P_\varepsilon(p)\times(-\delta, \delta)\to \mathbb{R}, \quad (x,t)\mapsto \beta(t).
\end{align}
Notice that on $U$
\begin{align}\label{eq: admissibility-test-fct}
    \nabla(\tilde\beta\circ\Psi^{-1})(y) \wedge g_0JN(y)=0.
\end{align}
In particular since $\partial_\nu u\wedge (N\circ u)=0$ on $\partial D^2$,
\begin{align}\label{eq: LHS-vanishes}
    \langle \partial_\nu u,I\nabla (\tilde{\beta}\circ \Psi^{-1})\circ u\rangle=0\text{ on } \sigma.
\end{align}
Let $\eta$ be a smooth cut-off function supported in $U$ and such that $\eta\equiv 1$ in a neighbourhood of
\begin{align*}
    \overline{\{x\in u(\sigma)\vert \tilde\beta\circ\Psi^{-1}(x)\neq 0\}}.
\end{align*}
By \eqref{eq: admissibility-test-fct} there holds
\begin{align}
    I\nabla(\eta(\tilde{\beta} \circ\Psi^{-1}))\cdot N=0
\end{align}
in a neighbourhood of $u(\sigma)$. Notice also that $\eta(\tilde{\beta} \circ\Psi^{-1})$ is a smooth, compactly supported function on $\mathbb{C}^2$, supported away from $u(\partial \sigma)$.
Thus by \eqref{eq: boundary-identity-I} and \eqref{eq: LHS-vanishes}
\begin{align}\label{eq: integral-G-zero}
    \int_{\partial D^2}\partial_\tau(\tilde{\beta}\circ\Psi^{-1}\circ u) G_\omega=\int_{\partial D^2}\partial_\tau((\eta(\tilde{\beta}\circ\Psi^{-1}))\circ u)  G_\omega=0.
\end{align}
We compute
\begin{align}
    \partial_\tau (\tilde{\beta}\circ\Psi^{-1}\circ u)=\partial_t\tilde{\beta}(\Psi^{-1}\circ u)\left\langle \left(\nabla\Psi^{-1}\right)\circ u\cdot \partial_\tau u, \frac{\partial}{\partial t}\right\rangle,
\end{align}
where $\cdot$ denotes the scalar product in the tangent bundle of $U$, while $\langle\cdot,\cdot\rangle$ denotes the scalar product in the tangent bundle of $B_\varepsilon^P(p)\times\{-\delta, \delta\}$.
Notice that the map
\begin{align}
    \pi: B^P_\varepsilon(p)\times(-\delta, \delta)\to\{p\}\times(-\delta, \delta),\quad (x,t)\mapsto (p,t)
\end{align}
is bi-Lipschitz when restricted to $(\Psi^{-1}\circ u)(\sigma)$
and
\begin{align}
    J\left(\Psi\circ\left(\pi\vert_{(\Psi^{-1}\circ u)(\sigma)}\right)^{-1}\right)=&\left(J\left(\pi\circ\left(\Psi\vert_{u(\sigma)}\right)^{-1}\right)\right)^{-1}\\
    \nonumber
    =&\left(\left\langle \left(\nabla\Psi^{-1}\right)\circ u\cdot \partial_\tau u, \frac{\partial}{\partial t}\right\rangle\right)^{-1}.
\end{align}
Therefore by means of a change of variable we can rewrite \eqref{eq: integral-G-zero} as
\begin{align}\label{eq: final-integral-G-zero}
    \int_{-\delta}^\delta \alpha(t)(G_\omega\circ u^{-1}\circ \Psi\circ(\pi\vert_{(\Psi^{-1}\circ u)^{-1})(\sigma)})(p, t)\, dt.
\end{align}
As \eqref{eq: final-integral-G-zero} holds for any smooth function $\alpha$ with zero average, we conclude that $G_\omega$ is constant on $\overline{\omega}\cap\partial D^2$. As observed at the end of Section 2, this implies $i\overline{g}\partial_\nu g=0$.\\

On the other hand assume that $g$ lies in $W^{1,1}(D^2)$ and satisfies $\operatorname{div}(i\overline{g}\nabla g)=0$ in $D^2$. Let $\omega\subset \overline{D^2}$ be a smooth, relatively open subset and let $f$ be a smooth function on $\mathbb{C}^2$ supported away from $u(\partial \omega \cap D^2)$ and such that $I\nabla f(x)\cdot N(x)=0$ for any $x$  in a neighbourhood of $u(\overline{\omega}\cap\partial D^2)$ in $\partial \Omega$. By computation \eqref{eq: computation-HSlp} we have
\begin{align}
    \int_{\omega}\langle du; d(I(\nabla f)\circ u)\rangle=\int_{\partial D^2\cap \omega}\langle \partial_\nu u, I(\nabla f)\circ u\rangle-\int_{\partial \omega}G_\omega \partial_\tau (f\circ u),
\end{align}
where $G_\omega$ is the function introduced in Remark \ref{rem: Neumann-distribution}. Notice that by the assumptions on $f$, in the second term of the right hand side the integrand vanishes on $\partial \omega\cap D^2$.
If $u$ satisfies \eqref{eq: orthogonality-prop} and \eqref{eq: bdry-condition-g}, the terms on the right hand side vanish (for the first term we use the fact that $I\nabla f(u(x))\cdot N(u(x))=0$ for any $x\in u(\overline{\omega}\cap\partial D^2)$, for the second we use the fact that $G_\omega$ is locally constant on $\overline{\omega}\cap \partial D^2$ by \eqref{eq: bdry-condition-g}), so that $u$ is $\Omega$-free boundary Hamiltonian stationary with the localization property.
\end{proof}
Theorem \ref{thm: HSLFB-implies-minimal} is a consequence of Theorem \ref{thm: properties of HSL finite sing}:
\begin{cor}\label{cor: lag-min-implies-min}
    Under the assumptions of Theorem \ref{thm: properties of HSL finite sing}, if we assume that the Lagrangian angle $\overline{g}$ is continuous, then $\overline{g}$ is constant and $u$ is a calibrated, branched $\Omega$-free boundary minimal immersion.
\end{cor}
\begin{proof}
    By Lemma \ref{lem: properties-of-g} the Lagrangian angle $\overline{g}$ of $u$ can be written as $\overline{g}=e^{-i\beta}$ for some harmonic function $\beta$ on $D^2$. By \eqref{eq: bdry-condition-g} in Theorem \ref{thm: properties of HSL finite sing}, $i\overline{g}\partial_\nu g=0$ on $\partial D^2$ in the sense described in Remark \ref{rem: Neumann-distribution}.
    We deduce that $\int_{D^2}\nabla \beta\cdot\nabla \varphi=0$ for any $\varphi\in C^1(\overline{D^2})$, which implies that $\beta$ is constant on $D^2$,
    say equal to $\beta_0\in \mathbb{R}$. Then $u$ is special Lagrangian. As $u$ is weakly conformal and satisfies $\Delta u=0$ (by \eqref{eq: structural-equation}), it is a branched minimal immersion. Moreover by Theorem \ref{thm: properties of HSL finite sing} it satisfies $\partial_\nu u\wedge (N\circ u)=0$ on $\partial D^2$, so that $u(\overline{D^2})$ is orthogonal to $\partial \Omega$ along $u(\partial D^2)$. We conclude that $u$ is a calibrated, branched $\Omega$-free boundary minimal immersion. 
\end{proof}
For the special case $\Omega=B_1(0)$ we obtain Theorem \ref{thm: main-thm}:
\begin{cor}
    Let $u\in C^1\cap W^{2,1}(\overline{D^2},\mathbb{C}^2)$ be a $B_1(0)$-free boundary Hamiltonian stationary Lagrangian immersion away fromm fiitely many branch points in $D^2$, with the localization property. Assume that the boundary curve satisfies $\partial_\tau u\cdot Iu =0$ (i.e. the boundary curve is Legendrian), then $u(\overline{D^2})$ is a Lagrangian flat equatorial disc (i.e. the intersection of $\overline{B_1(0)}$ with a Lagrangian 2-plane passing through the origin).
\end{cor}
\begin{proof}
    Corollary \ref{cor: lag-min-implies-min}  implies that  $u$ has constant Lagrangian angle, say equal to $e^{-i\beta_0}$ for some $\beta_0\in \mathbb{R}$, and is a weakly conformal branched $B_1(0)$-free boundary minimal immersion.\\
    Notice that by \eqref{eq: Lag-angle-def} and \eqref{eq: orthogonality-prop}
    \begin{align}
        \partial_\tau u=\pm\lvert\partial_\tau u\rvert e^{-i\beta_0}Ju\text{ on }\partial D^2.
    \end{align}
    Let $\gamma: [0,L)\to\mathbb{S}^3$ be an arclength parametrization of the curve $u\vert_{\partial D^2}$, then $\gamma$ satisfies $\dot{\gamma}=\pm e^{-i\beta_0}J\gamma$ and thus 
    parametrizes a great circle in $\mathbb{S}^3$ (i.e. the intersection of $\mathbb{S}^3$ with a 2-plane passing through the origin).
    Since $u$ satisfies \eqref{eq: structural-equation} with constant Lagrangian angle, $u$ is harmonic; thus by the maximum principle it takes values in the closed disc $\overline{D}$ spanned by $\gamma$. We conclude that $u(\overline{D^2})$ is a Lagrangian flat equatorial disc in $B_1(0)$.\\
    Alternatively, once we know that $u$ is a branched minimal immersion and satisfies $\partial_\nu u\wedge u=0$ on $\partial D^2$, the result follows from Theorem 2.1 in \cite{Fraser-Schoen} or Theorem 1.1 in \cite{LWW}. 
\end{proof}

\section{Examples of free boundary Hamiltonian stationary Lagrangian discs}
\subsection{Schoen-Wolfson cones}\label{section: SW}
In this section we show that Schoen-Wolfson cones are $B_1(0)$-free boundary Hamiltonian stationary Lagrangian surfaces.
This implies that the rigidity result of Theorem \ref{thm: main-thm} doesn't hold if we do not assume that the Lagrangian angle $\overline{g}$ is continuous.\\
Recall that for any relatively prime positive integers $p$ and $q$, the Schoen-Wolfson cone $\Sigma_{p,q}$--introduced in \cite{SW}-- 
has the following weakly conformal parametrization:
\begin{align}\label{eq: Phi-p-q}
\Phi_{p,q}:\overline{D^2}\to \mathbb{C}^2,\quad r e^{i\theta}\mapsto \frac{r^{\sqrt{pq}}}{\sqrt{p+q}} \begin{pmatrix}\sqrt{q}e^{ip\theta}
\\ i\sqrt{p}e^{-iq\theta}\end{pmatrix}.
\end{align}
The derivatives of $\Phi_{p,q}$ are given by
\begin{align}\label{eq: Phi-p-q-der}
\partial_r\Phi_{p,q}(r,\theta)=\frac{\sqrt{pq} \,r^{\sqrt{pq}-1}}{\sqrt{p+q}} \begin{pmatrix}\sqrt{q}e^{ip\theta}
\\ i\sqrt{p}e^{-iq\theta}\end{pmatrix},
\quad
\partial_\theta\Phi_{p,q}(r,\theta)=\frac{\sqrt{pq}\,r^{\sqrt{pq}}}{\sqrt{p+q}} \begin{pmatrix} i\sqrt{p}e^{ip\theta}
\\ \sqrt{q}e^{-iq\theta}
\end{pmatrix},
\end{align}
thus $\Phi_{p,q}$ is a Lagrangian map and its Lagrangian angle is given by
\begin{align}\label{eq: lag-angle-Phi-p-q}
    \frac{dz_1\wedge dz_2(\partial_r\Phi_{p,q}, \frac{1}{r}\partial_\theta\Phi_{p,q})}{\lvert\partial_r\Phi_{p,q}\wedge\frac{1}{r}\partial_\theta\Phi_{p,q}\rvert}=e^{i(p-q)\theta}.
\end{align}
\begin{lem}\label{lem: SW-cones}
    For any relatively prime positive integers $p,q$, the map $\Phi_{p,q}$ is $B_1(0)$-free boundary Hamiltonian stationary Lagrangian with the localization property.
\end{lem}
\begin{proof}
    First we notice that the Lagrangian angle $\overline{g}=e^{i(p-q)\theta}$ is $\mathbb{S}^1$-harmonic: observe that $i\overline{g}\nabla g=-\nabla^\perp \log (r^{p-q})$, therefore
    \begin{align}
      \operatorname{div}(i\overline{g}\nabla g)=-\operatorname{div}(\nabla^\perp \log (r^{p-q})) =0.
    \end{align}
    We also notice that $\Phi_{p,q}$ lies in $C^1\cap W^{2,1}(\overline{D^2})$, $\overline{g}$ lies in $ W^{1,(2,\infty)}$ and is smooth outside of the origin, and $\Phi_{p,q}(\partial D^2)\subset \mathbb{S}^3$.
    Therefore, in view of Theorem \ref{thm: properties of HSL finite sing}, in order to show that $\Phi_{p,q}$ is a $B_1(0)$-free boundary Hamiltonian stationary Lagrangian map with the localization property, it is enough to check conditions \eqref{eq: orthogonality-prop} and and \eqref{eq: bdry-condition-g}, i.e. that $i\overline{g}\partial_\nu g=0$ and $\partial_\nu\Phi_{p,q}\wedge\Phi_{p,q}=0$ on $\partial D^2$. The conditions can be verified directly from the explicit expressions \eqref{eq: Phi-p-q}, \eqref{eq: Phi-p-q-der} and \eqref{eq: lag-angle-Phi-p-q}.
\end{proof}

\subsection{Free boundary Hamiltonian stationary Lagrangian surfaces with non-constant angle}
The next result shows that if we remove the assumption that $u$ satisfies $\partial_\tau u\cdot I(N\circ u)=0$ on $\partial D^2$ from Theorem \ref{thm: HSLFB-implies-minimal}, then the Theorem might fail. The construction of the following example is based on observations made in \cite{GOR}.
\begin{lem}
    There exist smooth, conformal, Lagrangian maps with continuous Lagrangian angle which are $\Omega$-free boundary Hamiltonian stationary with the localization property for some smooth domain $\Omega$, but are not minimal.
\end{lem}
\begin{proof}
Let $\varphi: \overline{D^2}\to \mathbb{C}$ be defined by
\begin{align}
    \varphi(r e^{i\theta}):=r\cos(\theta)
\end{align}
and let $g:=e^{i\varphi}$. Then
\begin{align}
    \operatorname{div}(i\overline{g}\nabla g)=0,    
\end{align}
i.e. $g$ is $\mathbb{S}^1$-harmonic. Let
\begin{align}
    G(re^{i\theta}):= r\sin(\theta)
\end{align}
and notice that
\begin{align}
    \nabla^\perp G=i\overline{g}\nabla g.
\end{align}
Set
\begin{align}
    u=\begin{pmatrix}
        \overline{g}\\
        iG
    \end{pmatrix}.
\end{align}
Then $\operatorname{div}(g\nabla u)=0$, i.e. $u$ is a smooth conformal Hamiltonian stationary Lagrangian map with Lagrangian angle $\overline{g}$.
Observe that $u$ is a smooth embedding of a disc.
Consider the vector field
\begin{align}
    X:=\overline{g}J\partial_\tau u+G I\partial_\tau u
\end{align}
defined along $u(\partial D^2)$. Notice that $\partial_\tau u\cdot X=0$. Then there exists an open domain $\Omega$ with smooth boundary, such that $u(\partial D^2)\subset \partial\Omega$ and at any point $y$ of $u(\partial D^2)$, $N(y)\wedge X(y)=0$ (where $N$ denotes the outer normal vector of $\Omega$).\\
Let $\omega$ be a smooth, relatively open subset of $\overline{D^2}$. Let $f$ be a smooth function on $\mathbb{C}^2$ supported away from $u(\partial \omega\cap D^2)$ and such that $I\nabla f(x)\cdot N(x)=0$ for any $x$  in a neighbourhood of $u(\overline{\omega}\cap\partial D^2)$ in $\partial \Omega$. We compute as in \eqref{eq: computation-HSlp}
\begin{align}
    \int_\omega \langle du; d(I(\nabla f)\circ u)\rangle=&\int_{\partial\omega}\langle\partial_\nu u, I(\nabla f)\circ u\rangle-\int_{\omega}i\overline{g}dg\cdot d(f\circ u)\\
    \nonumber
    =&\int_{\overline{\omega}\cap\partial D^2} \langle\partial_\nu u, I(\nabla f)\circ u\rangle+\int_{\overline{\omega}\cap\partial D^2}\partial_\tau G\, f\circ u\\
    \nonumber
    =&\int_{\overline{\omega} \cap\partial D^2}\langle \overline{g}J\partial_\tau u+GI\partial_\tau u, I(\nabla f)\circ u\rangle\\
    \nonumber
    =&\int_{\overline{\omega} \cap\partial D^2}\langle X\circ u, I(\nabla f)\circ u\rangle=0,
\end{align}
where the last equality follows from the fact that $N\wedge X=0$ and $I(\nabla f)\cdot N=0$ on $u(\overline{\omega}\cap\partial D^2)$. We conclude that $u$ is a conformal, $\Omega$-free boundary Hamiltonian stationary Lagrangian map with the localization property, but its Lagrangian angle $\overline{g}$ is not constant.
\end{proof}


\begin{thebibliography}{10}

\bibitem{BBM}
J.~Bourgain, H.~Brezis, and P.~Mironescu.
\newblock {$H^{1/2}$} maps with values into the circle: minimal connections,
  lifting, and the {G}inzburg-{L}andau equation.
\newblock {\em Publ. Math. Inst. Hautes \'{E}tudes Sci.}, (99):1--115, 2004.

\bibitem{3-commutators}
F.~Da~Lio and T.~Rivi\`ere.
\newblock Three-term commutator estimates and the regularity of
  {$\frac12$}-harmonic maps into spheres.
\newblock {\em Anal. PDE}, 4(1), 2011.

\bibitem{EG}
Y.~Eliashberg and M.~Gromov.
\newblock Convex symplectic manifolds.
\newblock In {\em Several complex variables and complex geometry, {P}art 2
  ({S}anta {C}ruz, {CA}, 1989)}, volume 52, Part 2 of {\em Proc. Sympos. Pure
  Math.}, pages 135--162. Amer. Math. Soc., Providence, RI, 1991.

\bibitem{Fraser-Schoen}
A.~Fraser and R.~Schoen.
\newblock Uniqueness theorems for free boundary minimal disks in space forms.
\newblock {\em Int. Math. Res. Not. IMRN}, (17):8268--8274, 2015.

\bibitem{GOR}
F.~Gaia, G.~Orriols, and T.~Rivière.
\newblock A variational construction of hamiltonian stationary surfaces with
  isolated schoen-wolfson conical singularities.
\newblock {\em Comm. Pure Appl. Math.}, to appear.

\bibitem{LWW}
M.~Li, G.~Wang, and L.~Weng.
\newblock Lagrangian surfaces with {L}egendrian boundary.
\newblock {\em Sci. China Math.}, 64(7):1589--1598, 2021.

\bibitem{YS}
Y.~Luo and L.~Sun.
\newblock Rigidity theorems for minimal {L}agrangian surfaces with {L}egendrian
  capillary boundary.
\newblock {\em Adv. Math.}, 393:Paper No. 108124, 15, 2021.

\bibitem{Nitsche}
J.~Nitsche.
\newblock Stationary partitioning of convex bodies.
\newblock {\em Arch. Rational Mech. Anal.}, 89(1):1--19, 1985.

\bibitem{targetharm}
T.~Rivi\`ere.
\newblock The regularity of conformal target harmonic maps.
\newblock {\em Calc. Var. Partial Differential Equations}, 56(4):Paper No. 117,
  15, 2017.

\bibitem{Ros-Souam}
A.~Ros and R.~Souam.
\newblock On stability of capillary surfaces in a ball.
\newblock {\em Pacific J. Math.}, 178(2), 1997.

\bibitem{Schoen-special-lagrangian}
R.~Schoen.
\newblock Special {L}agrangian submanifolds.
\newblock In {\em Global theory of minimal surfaces}, volume~2 of {\em Clay
  Math. Proc.}, pages 655--666. Amer. Math. Soc., Providence, RI, 2005.

\bibitem{SW}
R.~Schoen and J.~Wolfson.
\newblock Minimizing area among {L}agrangian surfaces: the mapping problem.
\newblock {\em J. Differential Geom.}, 58(1):1--86, 2001.

\bibitem{SWsurvey}
R.~Schoen and J.~Wolfson.
\newblock The volume functional for {L}agrangian submanifolds.
\newblock In {\em Lectures on partial differential equations}, volume~2 of {\em
  New Stud. Adv. Math.}, pages 181--191. Int. Press, Somerville, MA, 2003.

\bibitem{Souam}
R.~Souam.
\newblock On stability of stationary hypersurfaces for the partitioning problem
  for balls in space forms.
\newblock {\em Math. Z.}, 224, 1997.

\end{thebibliography}
\end{document}